\documentclass{amsart}
\usepackage{amsmath}
\usepackage{amssymb}
\usepackage{amsthm}
\usepackage[top=3cm,left=3.5cm,right=3.5cm,bottom=3cm]{geometry} 
\usepackage{tgtermes}
\usepackage[T1]{fontenc}
\usepackage[shortlabels]{enumitem}
\usepackage[all]{xy}
\usepackage{color}
\usepackage{hyperref}
\theoremstyle{plain}
\newtheorem{theorem}{Theorem}[section]
\newtheorem{proposition}[theorem]{Proposition}
\newtheorem{lemma}[theorem]{Lemma}
\newtheorem{corollary}[theorem]{Corollary}
\theoremstyle{definition}
\newtheorem{definition}[theorem]{Definition}
\newtheorem{notation}[theorem]{Notation}
\newtheorem{construction}[theorem]{Construction}
\theoremstyle{remark}
\newtheorem{remark}[theorem]{Remark}
\newtheorem{example}[theorem]{Example}

\newcommand{\msf}[1]{\mathsf{#1}}
\newcommand{\mcal}[1]{\mathcal{#1}}
\newcommand{\mbb}[1]{\mathbb{#1}}

\newcommand{\mrm}[1]{\mathrm{#1}}

\newcommand{\Boxbar}{\overline{\Box}}
\DeclareMathOperator{\Hom}{Hom}

\DeclareMathOperator{\Gr}{Gr}
\DeclareMathOperator{\CH}{CH}
\DeclareMathOperator{\Pic}{Pic}
\DeclareMathOperator{\Spec}{Spec}
\DeclareMathOperator{\Proj}{Proj}
\DeclareMathOperator{\coker}{coker}
\DeclareMathOperator{\codim}{codim}
\DeclareMathOperator{\image}{image}
\DeclareMathOperator{\Fil}{Fil}

\DeclareMathOperator{\GL}{GL}
\DeclareMathOperator{\SL}{SL}
\DeclareMathOperator*{\colim}{colim}

\title{Isomorphisms up to bounded torsion between relative $K_0$-groups and Chow groups with modulus}
\author{Ryomei Iwasa}
\address{Department of Mathematical Sciences, University of Copenhagen, Universitetsparken 5, DK-2100 Copenhagen \O.}
\email{ryomei@math.ku.dk}
\author{Wataru Kai}
\address{Mathematical Institute, Tohoku University.
Aza-Aoba 6-3, Sendai 980-8578, Japan.}
\email{kaiw@tohoku.ac.jp}
\begin{document}

%\date{\today}
\maketitle

\tableofcontents

\section{Introduction}

The purpose of this note is to establish isomorphisms up to bounded torsion between relative $K_0$-groups and Chow groups with modulus as defined in \cite{BS17}.

\begin{theorem}\label{mainthm}
Let $X$ be a separated regular noetherian scheme of dimension $d$ and $D$ an effective Cartier divisor on $X$.
Assume that $D$ has an affine open neighborhood in $X$.
Then there exists a finite descending filtration $F^*$ on $K_0(X,D)$ and, for each integer $p$, there exists a surjective group morphism
\[
	\msf{cyc}\colon \CH^p(X|D) \twoheadrightarrow F^pK_0(X,D)/F^{p+1}K_0(X,D)
\]
such that its kernel is $(p-1)!^N$-torsion for some positive integer $N$ depending only on $p$.
Furthermore, the filtration $F^*$ coincides with the gamma filtration on $K_0(X,D)$ up to $(d-1)!^M$-torsion for some positive integer $M$ depending only on $d$.
\end{theorem}

The case $D=\varnothing$ is a classical theorem of Soul\'e \cite{So85}, which owes its origin to Grothendieck's Riemann-Roch type formula \cite{SGA6}.
The filtration $F^*$ and the morphism $\msf{cyc}$ have been constructed in \cite{Iw19} in a slightly weaker generality.
The assumption that $D$ has an affine open neighborhood is essential, see Example \ref{counterexample}.

Let $\mcal{S}(X|D)$ be the set of all closed subsets in $X$ not meeting $D$ and $\mcal{S}(X|D,1)$ the set of all closed subsets in $X\times\Box^1$ satisfying the modulus condition along $D$.
It follows easily from the definition of $\CH^*(X|D)$ that there is an exact sequence
\[
\xymatrix@1{
	\displaystyle\colim_{Y\in\mcal{S}(X|D,1)}\mcal{Z}^*_Y(X\times\Box^1) \ar[r] & \displaystyle\colim_{Y\in\mcal{S}(X|D)}\CH^*_Y(X) \ar[r] & \CH^*(X|D) \ar[r] & 0,
}
\]
where $\mcal{Z}^*_Y(-)$ is the group of cycles with supports in $Y$ and $\CH^*_Y(-)$ is the Chow group with supports in $Y$.
The real content of this note is to establish an analogous exact sequence for $K$-groups.
In the second section, as Theorem \ref{thm:K-theory}, we establish an exact sequence
\[
\xymatrix@1{
	\displaystyle\colim_{Y\in\mcal{S}(X|D,1)}K_0^Y(X\times\Box^1) \ar[r] & \displaystyle\colim_{Y\in\mcal{S}(X|D)}K_0^Y(X) \ar[r] & K_0(X,D) \ar[r] & 0.
}
\]
Then, from the classical rational isomorphisms between $K_0$-groups and Chow groups, we get a rational isomorphism between $K_0(X,D)$ and $\CH^*(X|D)$.
The estimate on torsion is obtained by using Adams decomposition.

\subsection*{Convention}
All rings are noetherian and all schemes are separated noetherian.
For a point $v$ of a scheme $X$, we denote by $\kappa(v)$ the residue field of $v$.

\section{A presentation of Chow group with modulus}\label{cycle}

\subsection*{Chow groups with supports}

%\cite[8.1]{GS87}.

\begin{notation}
Let $X$ be a scheme and $p$ an integer.
\begin{enumerate}[(1)]
\item We write $X^{(p)}$ for the set of all points of codimension $p$ in $X$, i.e., points $v\in X$ whose closures in $X$ have codimension $p$.
We understand $X^{(p)}=\varnothing$ if $p<0$.
\item For a closed subset $Y$ of $X$, we define $\mcal{Z}^p_Y(X)$ to be the free abelian group with the generators $[V]$, one for each $v\in X^{(p)}\cap Y$, with $V$ being the closure of $v$ in $X$.
We write $\mcal{Z}^p(X)=\mcal{Z}^p_X(X)$.
\item For a closed subscheme $D$ of pure codimension $p$ in $X$, we write
\[
	[D] := \sum_{x_i\in D^{(0)}}\mrm{length}(\mcal{O}_{D,x_i})[D_i] \in \mcal{Z}^p_D(X),
\]
where $D_i$ is the closure of $x_i$ in $X$.
\end{enumerate}
\end{notation}

\begin{construction}
Let $X$ be a scheme and $p$ an integer.
Let $w\in X^{(p-1)}$ and write $W$ for its closure in $X$.
For each $v\in X^{(p)}\cap W$, there exists a unique group morphism $\nu_v\colon \kappa(w)^\times \to \mbb{Z}$ which sends $a\in\mcal{O}_{W,v}\setminus\{0\}$ to the length of $\mcal{O}_{W,v}/(a)$.
For $f\in \kappa(w)^\times$, we define
\[
	\mrm{div}(f) := \sum_{v\in X^{(p)}\cap W}\nu_v(f)[V] \in \mcal{Z}^p_W(X).
\]
\end{construction}

\begin{definition}\label{def:Chow_sup}
Let $X$ be a scheme and $p$ an integer.
For a closed subset $Y$ of $X$, we define
\[
	\CH^p_Y(X) := \coker\Bigl(\bigoplus_{w\in X^{(p-1)}\cap Y}\kappa(w)^\times \xrightarrow{\mrm{div}} \mcal{Z}_Y^p(X)\Bigr).
\]
We write $\CH^p(X)=\CH^p_X(X)$.
\end{definition}

\begin{definition}\label{def:unicodim}
Let $X$ be a topological space with irreducible components $\{X_i\}_{i\in I}$.
We say that $X$ is \textit{unicodimensional} if $\codim_X(V)=\codim_{X_i}(V)$ for any $i\in I$ and any irreducible closed subset $V$ of $X_i$.
A scheme is unicodimensional if the underlying topological space is unicodimensional.
\end{definition}

%\begin{remark}
%If a scheme $X$ is biequidimensional in the sense of \cite[1.2 (ii)]{He} with $d=\dim X<\infty$, then $X^{(p)}=X_{(d-p)}$ the set of all points of dimension $(d-p)$.
%A space is biequidimensional if and only if it is equidimensional, equicodimensional, unicodimensional and catenary.
%\end{remark}

\begin{lemma}\label{lem:pushfoward}
Let $X$ be a unicodimensional catenary scheme, $D$ a closed subscheme of pure codimension $r$ in $X$ and $p$ an integer.
Then $D$ is unicodimensional and $D^{(p-r)}\subset X^{(p)}$.
Furthermore, the inclusion $\iota\colon D\hookrightarrow X$ induces a group morphism
\[
	\iota_*\colon \CH^{p-r}_{D\cap Y}(D) \to \CH^p_Y(X)
\]
for any closed subset $Y$ of $X$.
\end{lemma}
\begin{proof}
Let $\{X_i\}_{i\in I}$ (resp.\ $\{D_j\}_{j\in J}$) be the set of irreducible components of $X$ (resp.\ $D$).
Take $j\in J$ and $v\in D^{(p-r)}\cap D_j$.
Then 
\[
	\codim_{D_j}(v) + \codim_{X_i}(D_j) = \codim_{X_i}(v) = \codim_X(v)
\]
for any $i\in I$ with $X_i\supset D_j$.
Since $\codim_{X_i}(D_j)=r$ regardless of the choices of $i,j$, we see that $D$ is unicodimensional and that $\codim_X(v)=\codim_D(v)+r=p$.
Hence, $D^{(p-r)}\subset X^{(p)}$.
The last statement is immediate from this.
\end{proof}

\begin{lemma}\label{lem:pullback}
Let $X$ be a unicodimensional catenary scheme, $Y$ a closed subset of $X$, $D$ an effective Cartier divisor on $X$ and $p$ an integer.
Let $v\in X^{(p)}\cap Y$ whose closure $V$ in $X$ is not contained in $D$.
Then $[V\times_XD]\in \mcal{Z}^p_{D\cap Y}(D)$.
\end{lemma}
\begin{proof}
It suffices to show that $\codim_D(V\times_XD)=p$.
First of all, note that $V\times_XD$ is an effective Cartier divisor on $V$, and thus it is of pure codimension $1$ in $V$.
It follows that $\codim_X(V\times_XD)=p+1$.
Since $X$ is catenary, we conclude that $\codim_D(V\times_XD)=\codim_X(V\times_XD)-\codim_X(D)=p$.
\end{proof}

\begin{construction}\label{constr:pullback}
Let $X$ be a unicodimensional catenary scheme, $Y$ a closed subset of $X$, $D$ an effective Cartier divisor on $X$ and $p$ an integer.
We define a group morphism
\[
	\iota^*\colon \mcal{Z}^p_Y(X) \to \CH^p_{Y\cap D}(D)
\]
as follows, where $\iota$ refers to the inclusion $D\hookrightarrow X$.
For an integral closed subscheme $V$ of codimension $p$ in $X$ whose support is in $Y$,
\[
\iota^*([V]) := 
\begin{cases}
	[V\times_XD] & \text{if }V\nsubseteq D \\
	j_*[\mcal{O}_X(D)\vert_V] & \text{if }V\subseteq D
\end{cases}\in\CH_{Y\cap D}^p(D),
\]
where the first equation is well-defined by Lemma \ref{lem:pullback} and, for the second, $j$ refers to the inclusion $V\hookrightarrow D$ and $j_*\colon \CH^1(V) \to \CH_{Y\cap D}^p(D)$ is the push-forward ensured by Lemma \ref{lem:pushfoward}.
\end{construction}

\begin{remark}
It is the classical fact that the morphism $\iota^*$ in Construction \ref{constr:pullback} factors through $\CH^p_Y(X)$ if $X$ is an algebraic scheme, cf., \cite[Chapter 2]{Fu98}.
It would be true more generally, but we do not need such a result for our purpose.
\end{remark}

\subsection*{Chow groups with modulus}

\begin{notation}
We set $\Boxbar^1 := \Proj(\mbb{Z}[T_0,T_1])$ and let $t$ be the rational coordinate $T_0/T_1$.
We write $\Box^1:=\Boxbar^1\setminus(t=\{\infty\})$.
For an integer $q$ and a scheme $X$, we denote by $\iota_{X,q}$ (or simply by $\iota_q$) the inclusion $X \hookrightarrow X\times\Box^1$ defined by $t=q$.
\end{notation}

\begin{definition}[Binda-Saito]\label{def:modulus_cond}
Let $X$ be a scheme and $D$ an effective Cartier divisor on $X$.
Let $W$ be a closed subset of $X\times\Box^1$.
Let $\overline{W}^N$ be the normalization of the closure $\overline{W}$ (with the reduced scheme-structure) of $W$ in $X\times\Boxbar^1$ and denote by $\phi_W$ the canonical morphism $\overline{W}^N\to\overline{W}$.
We say that \textit{$W$ satisfies the modulus condition along $D$} if the following inequality of Cartier divisors on $\overline{W}^N$ holds
\[
	\phi_W^*(D\times\Boxbar^1)\le \phi_W^*(X\times\{\infty\}).
\]
\end{definition}

\begin{lemma}\label{lem:containment}
Let $X$ be a scheme and $D$ an effective Cartier divisor on $X$.
Let $W$ be a closed subset of $X\times\Box^1$ satisfying the modulus condition along $D$.
Then any closed subset of $W$ satisfies the modulus condition along $D$.
\end{lemma}
\begin{proof}
The proof for \cite[Proposition 2.4]{KP} works, noting that every integral morphism of schemes is closed (the theorem of Cohen-Seidenberg, \cite[(6.1.10)]{EGA2} or \cite[Theorem 5.10]{Ati-Mac}).

Here we give an alternative argument which might %will 
be useful later.
First, since the modulus condition is a local condition, we may assume $X$ is affine $X=\Spec (A)$ and $D$ is principal $D=(f)$.
We give $\overline{W}\subset X\times\Boxbar^1$ the reduced scheme structure and 
consider its restriction to the open subset $ X\times (\Boxbar^1\setminus \{ 0\} )=\Spec (A[1/t])$.
Let $ A[t^{-1}]/J$ be the coordinate ring of $\overline{W}\cap\Spec (A[1/t]) $.
The modulus condition for $W$ is equivalent to the condition that the element $1/tf$ in the ring of total quotients of $A[1/t]/J$ is integral over this ring, i.e., 
that there is a relation in $A[1/t]/J$ of the form
\[
	\frac{1}{t^n} + g_1\frac{1}{t^{n-1}} + \dotsb + g_n = 0
\]
with $g_i\in f^iA[1/t]/J$.
Now let $Y\subset W$ be any closed subset. 
Then we have the image of the above relation to the coordinate ring of $\overline{Y}\cap\Spec (A[1/t])$.
This implies the modulus condition for $Y$.
%(Note that this argument proves the corollary below.)
\end{proof}

%\begin{corollary}\label{cor:containment}
%In the notation of Lemma \ref{lem:containment}, if $X$ is affine and $D$ is principal, then any closed subset $W\subset X\times \square ^1$ satisfying the modulus condition is contained in a closed subset of purely codimension $1$ in $X\times \square ^1$ which satisfies the modulus condition.
%\end{corollary}

\begin{notation}\label{not:modulus}
Let $X$ be scheme, $D$ an effective Cartier divisor on $X$ and $p$ an integer.
\begin{enumerate}[(1)]
\item $\mcal{S}(X|D)$ is the set of all closed subsets of $X$ not meeting $D$.
\item $\mcal{S}(X|D,1)$ is the set of all closed subsets of $X\times\Box^1$ satisfying the modulus condition along $D$.
\item $\mcal{Z}^p(X|D)$ is the free abelian group with generators $[V]$, one for each $v\in X^{(p)}$ whose closure $V$ does not meet $D$.
\item $\mcal{Z}^p(X|D,1)$ is the free abelian group with generators $[W]$, one for each $w\in (X\times\Box^1)^{(p)}$ whose closure $W$ is dominant over $\Box^1$ and satisfies the modulus condition along $D$.
\end{enumerate}
\end{notation}

\begin{remark}\label{rem:modulus}
We remark that 
\[
	\mcal{Z}^p(X|D)=\colim_{Y\in\mcal{S}(X|D)}\mcal{Z}^p_Y(X)
	\quad \text{and} \quad
	\mcal{Z}^p(X|D,1)\subset \colim_{Y\in\mcal{S}(X|D,1)}\mcal{Z}^p_Y(X\times\Box^1).
\]
In the latter formula, the difference consists of cycles not dominant over $\Box^1$.
\end{remark}

\begin{definition}\label{def:Chow_modulus}
Let $X$ be a unicodimensional catenary scheme, $D$ an effective Cartier divisor on $X$ and $p$ an integer.
We define
\[
	\CH^p(X|D) := \coker\bigl(\mcal{Z}^p(X|D,1)\xrightarrow{\iota_0^*-\iota_1^*}\mcal{Z}^p(X|D)\bigr),
\]
where the morphisms $\iota_0^*$ and $\iota_1^*$ are well-defined by Lemma \ref{lem:pullback} and Remark \ref{rem:modulus}.
\end{definition}

%\textcolor{magenta}{\textbf{Comment.}
%$\CH^p(X)=\CH^p(X|\varnothing)$?
%The Lemma \ref{lem:epsilon} below implies that there is a surjection $\CH^p(X)\twoheadrightarrow\CH^p(X|\varnothing)$.}

\begin{lemma}\label{lem:epsilon}
Let $X$ be a unicodimensional catenary scheme, $D$ an effective Cartier divisor on $X$, $Y$ a closed subset of $X$ not meeting $D$ and $p$ an integer.
Then the canonical morphism
\[
	\mcal{Z}^p_Y(X) \to \CH^p(X|D)
\]
factors through $\CH^p_Y(X)$.
\end{lemma}
\begin{proof}
Suppose given $w\in X^{(p-1)}\cap Y$ and $f\in\kappa(w)^\times$.
We have to show that $\mrm{div}(f)=0$ in $\CH^p(X|D)$.
Consider the Cartier divisor $E$ on $W\times\Box^1$ defined by $f+(1-f)t$.
%This corresponds to the graph of the rational map $f\colon W\to\Box^1$.
Then $E$ gives an element $[E]$ in $\mcal{Z}^p(X|D,1)$ and $\iota_0^*[E]-\iota_1^*[E]=\mrm{div}(f)$ in $\mcal{Z}^p(X|D)$.
This prove the lemma.
\end{proof}

\begin{proposition}\label{prop:cycle}
Let $X$ be a unicodimensional catenary scheme, $D$ an effective Cartier divisor on $X$ and $p$ an integer.
Then the sequence 
\[
\xymatrix@1{
	\displaystyle\colim_{Y\in\mcal{S}(X|D,1)}\mcal{Z}^p_Y(X\times\Box^1) \ar[r]^-{\iota_0^*-\iota_1^*} 
	& \displaystyle\colim_{Y\in\mcal{S}(X|D)}\CH^p_Y(X) \ar[r]^-{\epsilon} & \CH^p(X|D) \ar[r] & 0
}
\]
is exact.
Here, $\iota_0^*,\iota_1^*$ are the morphisms defined in Construction \ref{constr:pullback}, and $\epsilon$ is the canonical morphism as in Lemma \ref{lem:epsilon}.
\end{proposition}
\begin{proof}
We only have to show that the composite $\epsilon\circ(\iota_0^*-\iota_1^*)$ is zero.
Let $Y\in\mcal{S}(X|D,1)$ and $v\in(X\times\Box^1)^{(p)}\cap Y$.
Then the closure $V$ of $v$ satisfies the modulus condition along $D$ by Lemma \ref{lem:containment}.
If $v\notin X\times\{0,1\}$, then it is immediate from the definition of $\CH^p(X|D)$ that $(\epsilon\circ(\iota_0^*-\iota_1^*))([V])=0$.
If $v\in X\times\{0,1\}$, say $v\in X\times\{0\}$, then $\iota_1^*[V]=0$ and $\iota_0^*[V]=j_*[\mcal{O}_{X\times\Box^1}(X)\vert_V]=0$ since $X$ is a principal divisor in $X\times\Box^1$.
This completes the proof.
\end{proof}

\section{A presentation of relative $K_0$-group}\label{K-theory}

For a scheme $X$, we denote by $K(X)$ Thomason-Trobaugh's $K$-theory spectrum \cite[3.1]{TT90}.
The spectrum $K(X)$ is contravariant functorial in $X$.

\begin{definition}
Let $X$ be a scheme.
\begin{enumerate}[(1)]
\item Let $D$ be a closed subscheme of $X$.
We define $K(X,D)$ to be the homotopy fiber of the canonical morphism $K(X)\to K(D)$.
For an integer $n$, we write $K_n(X,D)=\pi_nK(X,D)$.
\item Let $Y$ be a closed subset of $X$.
We define $K^Y(X)$ to be the homotopy fiber of the canonical morphism $K(X)\to K(X\setminus Y)$.
For an integer $n$, we write $K_n^Y(X)=\pi_nK^Y(X)$.
\end{enumerate}
\end{definition}

The goal of this section is to prove the following theorem.
\begin{theorem}\label{thm:K-theory}
Let $X$ be a regular scheme and $D$ an effective Cartier divisor on $X$.
Assume that $D$ admits an affine open neighborhood in $X$.
Then the sequence
\[
\xymatrix@1{
	\displaystyle\colim_{Y\in\mcal{S}(X|D,1)}K_0^Y(X\times\Box^1) \ar[r]^-{\iota_0^*-\iota_1^*} 
	& \displaystyle\colim_{Y\in\mcal{S}(X|D)}K_0^Y(X) \ar[r]^-{\epsilon} & K_0(X,D) \ar[r] & 0
}
\]
is exact.
Here, $\epsilon$ denotes the obvious morphism.
\end{theorem}

The surjectivity of $\epsilon$ has been observed in \cite{Iw19}.
\begin{lemma}\label{lem:surjectivity}
Let $X$ be a scheme and $D$ a closed subscheme of $X$.
Assume that $D$ has an affine open neighborhood in $X$.
Then the canonical morphism
\[
	\epsilon\colon \colim_{Y\in\mcal{S}(X|D)}K_0^Y(X) \to K_0(X,D)
\]
is surjective.
\end{lemma}
\begin{proof}
Let $U$ be an affine open neighborhood of $D$ in $X$.
By the localization theorem \cite[7.4]{TT90}, the sequence
\[
\xymatrix@1{
	K_0^{X\setminus U}(X) \ar[r] & K_0(X,D) \ar[r] & K_0(U,D)
}
\]
is exact.
Hence, we may replace $X$ by $U$ and reduce to the case $X$ is affine.
Then the result follows from \cite[Lemma 3.4]{Iw19}.
\end{proof}

\subsection*{The rigidity}

\begin{lemma}\label{lem:rigidity}
Let $X$ be a regular scheme and $D$ an effective Cartier divisor.
Let $Y\in\mcal{S}(X|D,1)$ and denote its closure in $X\times\Boxbar^1$ by $\overline{Y}$.
Assume that $D$ admits an affine open neighborhood in $X$.
Then the two morphisms
\[
	\iota_0^*,\iota_1^*\colon K_0^{\overline{Y}}(X\times\Boxbar^1) \to K_0(X,D)
\]
coincide.
\end{lemma}

\begin{proof}
First of all, let us fix notation for morphisms of schemes:
\[
\xymatrix{
	X \ar@<0.5ex>[r]^-{\iota_0} \ar@<-0.5ex>[r]_-{\iota_1} & X\times\Boxbar^1 \ar[d]^p \ar[r]^-q & X & D \ar[l]_-i \\
	& \Boxbar^1 & & 
}
\]
where $p,q$ are the canonical projections and $i$ is the canonical inclusion.

According to \cite[Theorem 3.1]{Iw19}, $K_0(X,D)$ is generated by triples $(P,\alpha,Q)$ where $P,Q$ are perfect complexes of $X$ and $\alpha$ is a quasi-isomorphism $Li^*P \xrightarrow{\sim} Li^*Q$.
The morphism 
\[
	\iota_a^* \colon K_0^{\overline{Y}}(X\times\Boxbar^1) \to K_0(X,D) \quad a\in\{0,1\}
\]
sends $[P]$ to $[(L\iota_a^*P,0,0)]$, where $P$ is a perfect complex of $X\times\Boxbar^1$ whose support lies in $\overline{Y}$.
Since $X$ is regular, $K_0^{\overline{Y}}(X\times\Boxbar^1)$ is generated by coherent $\mcal{O}_{\overline{Y}}$-modules.
Hence, it suffices to show that, for any coherent $\mcal{O}_{\overline{Y}}$-modules $\mcal{F}$,
\[
	[(L\iota_0^*\mcal{F},0,0)]=[(L\iota_1^*\mcal{F},0,0)]
\]
in $K_0(X,D)$.
In the sequel, we denote by $\mcal{F}$ a coherent $\mcal{O}_{\overline{Y}}$-module.

\subsubsection*{First calculation in $K_0(X,D)$}
Recall that we have fixed a rational coordinate $t$ of $\Boxbar^1$.
We denote by $\mcal{O}(-1)$ the invertible sheaf on $\Boxbar^1$ generated by $t$.
We write $j_0$ for the canonical inclusion $\mcal{O}(-1)\to\mcal{O}_{\Boxbar^1}$ sending $t$ to $t$, and write $j_1$ for the inclusion $\mcal{O}(-1)\to\mcal{O}_{\Boxbar^1}$ sending $t$ to $t-1$.
Then we have an exact triangle
\[
\xymatrix@1{
	Lp^*\mcal{O}(-1)\otimes_{X\times\Boxbar^1}^L\mcal{F} \ar[r]^-{p^*j_a} & \mcal{F} \ar[r] & \iota_{a*}L\iota_a^*\mcal{F} \ar[r]^-+ &
}
\]
of perfect complexes of $X\times\Boxbar^1$ for $a\in\{0,1\}$.
Consequently,
\[
	[(L\iota_a^*\mcal{F},0,0)] = -Rq_*[(Lp^*\mcal{O}(-1)\otimes_{X\times\Boxbar^1}^L\mcal{F},p^*j_a,\mcal{F})]
\]
in $K_0(X,D)$.
We set $\theta:=(t-1)/t$.
Then we have a commutative diagram (the vertical arrow is defined after restricting to $\Boxbar^1\setminus\{0\}$)
\[
\xymatrix@R-1pc{
	& \mcal{O}_{\Boxbar^1} \ar@{.>}[dd]^{\text{multiplication by }\theta} \\
	\mcal{O}(-1) \ar[ur]^{j_0} \ar[dr]_{j_1} & \\
	& \mcal{O}_{\Boxbar^1}.
}
\]
Since $Y$ satisfies the modulus condition, the multiplication by $\theta$ on $\mcal{O}_{\overline{Y}}$-modules makes sense in a neighborhood of $\overline{Y}\cap(D\times\Boxbar^1)$.
It follows from the above diagram that 
\[
	[(Lp^*\mcal{O}(-1)\otimes_{X\times\Boxbar^1}^L\mcal{F},p^*j_0,\mcal{F})] + [(\mcal{F},\theta,\mcal{F})]
	= [(Lp^*\mcal{O}(-1)\otimes_{X\times\Boxbar^1}^L\mcal{F},p^*j_1,\mcal{F})].
\]
Hence, it remains to show that
\[
	Rq_*[(\mcal{F},\theta,\mcal{F})]=0
\]
in $K_0(X,D)$.

\subsubsection*{Adic filtration on $\mcal{F}$}
Let $\Fil^*\mcal{F}$ be the adic filtration on $\mcal{F}$ with respect to the ideal defining $D\times\{\infty\}$ in $X\times\Boxbar^1$.
Since $(1/t)\Fil^l\mcal{F}\subset \Fil^{l+1}\mcal{F}$ in a neighborhood of $\overline{Y}\cap(D\times\Boxbar^1)$, we see that
\[
	[(\Fil^l\mcal{F}/\Fil^{l+1}\mcal{F},\theta,\Fil^l\mcal{F}/\Fil^{l+1}\mcal{F})] = 0
\]
for all $l\ge 0$.
Hence, we are reduced to showing that $Rq_*[(\Fil^l\mcal{F},\theta,\Fil^l\mcal{F})]=0$ for some $l\ge 0$.
We show that $\theta$ acts on $Li^*Rq_*\Fil^l\mcal{F}$ as the identity for sufficiently large $l$.

Take an affine open neighborhood $U$ of $D$ in $X$ such that $\overline{Y}_U:=\overline{Y}\times_XU$ misses $X\times\{0,1\}$ and that the restriction $q|_{\overline{Y}_U}\colon \overline{Y}_U\to U$ is finite.
We set some notation:
\[
	A=\mcal{O}(U) \qquad A/I=\mcal{O}(D) \qquad A[1/t]/J=\mcal{O}(\overline{Y}_U) \qquad M=\mcal{F}(\overline{Y}_U)
\]
The filtration $\Fil^*\mcal{F}$ on $\mcal{F}$ descends to a filtration $\Fil^*M$ on $M$, which is identified with the $(I,1/t)$-adic filtration.
Observe that $Li^*Rq_*\Fil^l\mcal{F} = (Li^*\Fil^lM)^{\sim}$.

We claim that there exists $n\ge 0$ such that $\Fil^{l+1}M=I\Fil^lM$ and $H_k(Li^*\Fil^lM)=0$ for all $l\ge n$ and $k>0$.
If we admit the claim, then
\[
	Li^*\Fil^lM = \Fil^lM/I\Fil^lM = \Fil^lM/\Fil^{l+1}M
\]
on which we know that $\theta$ acts as the identity. 
The claim is a local question on $\Spec A$, and thus we may assume that $I$ is principal, $I=(f)$ with $0\ne f\in A$.
By the modulus condition, we have a relation in $A[1/t]/J$ of the form
\[
	\frac{1}{t^n} + g_1\frac{1}{t^{n-1}} + \dotsb + g_{n-1}\frac{1}{t} + g_n = 0
\]
for some $n\ge 0$ and $g_k\in f^kA[1/t]/J$ with $1\le k\le n$.
Repeated application of this relation gives
\[
	\Fil^lM = f^lM + \frac{f^{l-1}}{t}M + \dotsb + \frac{f^{l-(n-1)}}{t^{n-1}}M
\]
for $l\ge n$.
In particular, $\Fil^{l+1}M=f\Fil^lM$.
Furthermore, since the $f$-power torsion of $M$ has a bounded exponent, $\Fil^lM=f^{l-n}\Fil^nM$ has no $f$-power torsion for $l\gg n$.
This proves the claim.
\end{proof}

\begin{corollary}\label{cor:rigidity}
Under the situation in Theorem \ref{thm:K-theory}, $\epsilon\circ(\iota_0^*-\iota_1^*)=0$.
\end{corollary}
\begin{proof}
Since $X$ is regular, the restriction morphism
\[
	K_0^{\overline{Y}}(X\times\Boxbar^1) \to  K_0^Y(X\times\Box^1)
\]
is surjective, where $Y\in\mcal{S}(X|D,1)$ and $\overline{Y}$ is the closure of $Y$ in $X\times\Boxbar^1$.
Hence, the result follows from Lemma \ref{lem:rigidity}.
\end{proof}

\subsection*{End of the proof}

\begin{lemma}\label{lem:K_1}
Let $X$ be a scheme and $D$ an effective Cartier divisor on $X$ admitting an affine open neighoborhood in $X$.
Suppose we are given $Z\in\mcal{S}(X|D)$ and $\alpha\in K_1(X\setminus Z)$ whose restriction to $K_1(D)$ is zero.
Then there exist $W\in\mcal{S}(X|D,1)$ and $\beta\in K_1(X\times\Box^1\setminus W)$ such that
\[
	\beta\vert_{t=0}-\beta\vert_{t=1}=\alpha \quad \text{in} \quad \colim_{Y\in\mcal{S}(X|D)}K_1(X\setminus Y).
\]
\end{lemma}
\begin{proof}
By enlarging $Z$ if necessarily, we may assume that $X\setminus Z$ is affine.
Set $U=X\setminus Z$.
Take a representative of $\alpha$ in $\GL_n(U)$, which we also denote by $\alpha$.
By our assumption, the restriction of $\alpha$ to matrices over $D$ is in the group $E_m(D)$ of elementary matrices for some $m\ge n$.
Take a lift $\alpha'\in E_m(U)$ of $\alpha\vert_D$.
Then $\alpha = \alpha' + \epsilon$ in $\GL_m(U)$, where $\epsilon$ is a matrix whose entries are all in the ideal $I$ defining $D$.
We define an $(m\times m)$-matrix over $U\times\Box^1$ by
\[
	\alpha(t) := \alpha' + (1-t)\epsilon.
\]
Then the determinant $\det(\alpha(t))$ is an admissible polynomial for $D$ in the sense \cite[\S4]{BS17}, and thus its zero locus $W=V(\det(\alpha(t)))$ satisfies the modulus condition along $D$.
Lastly, by definition, $\alpha(t)$ gives an element $\beta$ in $K_1(X\times\Box^1\setminus W)$ and it satisfies the desired formula.
\end{proof}

\begin{proof}[Proof of Theorem \ref{thm:K-theory}]
By Lemma \ref{lem:surjectivity} and Corollary \ref{cor:rigidity}, it remains to show the exactness at the middle term.
Suppose we are given $Z\in\mcal{S}(X|D)$ and $\alpha\in K_0^Z(X)$ which dies in $K_0(X,D)$ along the obvious morphism.
Watch the commutative diagram with exact rows
\[
\xymatrix{
	K_1(X\setminus Z,D) \ar[r] \ar@{=}[d] & K_1(X\setminus Z) \ar[r] \ar[d] & K_1(D) \ar[d] \\
	K_1(X\setminus Z,D) \ar[r] & K_0^Z(X) \ar[r] & K_0(X,D).
}
\]
It follows that there exists a lift $\alpha'\in K_1(X\setminus Z)$ of $\alpha$ along the boundary morphism whose restriction to $K_1(D)$ is zero.
Hence, by Lemma \ref{lem:K_1}, we find an element $\beta'$ in the group at the left upper corner of the commutative diagram
\[
\xymatrix{
	\displaystyle\colim_{Y\in\mcal{S}(X|D,1)}K_1(X\times\Box^1\setminus Y) \ar[r]^-{\iota_0^*-\iota_1^*} \ar[d]
		& \displaystyle\colim_{Y\in\mcal{S}(X|D)}K_1(X\setminus Y) \ar[d] & \\
	\displaystyle\colim_{Y\in\mcal{S}(X|D,1)}K_0^Y(X\times\Box^1) \ar[r]^-{\iota_0^*-\iota_1^*} 
		& \displaystyle\colim_{Y\in\mcal{S}(X|D)}K_0^Y(X) \ar[r]^-{\epsilon} & K_0(X,D)
}
\]
such that $\iota_0^*\beta'-\iota_1^*\beta'=\alpha'$ in the upper middle group.
This proves the exactness of the lower sequence.
\end{proof}

\section{Adams decomposition}\label{adams}

\begin{notation}\label{not:w_i}
For $i\ge 0$, we set
\[
	w_i = 
	\begin{cases} 
		1 & \text{if $i$ is zero} \\
		2 & \text{if $i$ is odd} \\
		\text{denominator of }|B_{i/2}|/2i & \text{if $i$ is positive even}
	\end{cases}
\]
where $B_n$ denotes the $n$-th Bernoulli number.
\end{notation}

\begin{lemma}
\leavevmode
\begin{enumerate}[label={\upshape(\roman*)}]
\item If a prime $p$ divides $w_i$, then $(p-1)$ divides $i$.
The converse is true if $i$ is even.
\item For $i\ge 1$ and for $N$ large enough than $i$,
\[
	w_i = \gcd_{k\ge 2}k^N(k^i-1).
\]
\end{enumerate}
\end{lemma}
\begin{proof}
See \cite[Appendix B]{MS74}.
\end{proof}

\begin{definition}
Let $N$ be a positive integer.
We say that a morphism $f\colon A\to B$ of abelian groups is an \textit{$N$-monomorphism} (resp.\ \textit{$N$-epimorphism}) if the kernel (resp.\ cokernel) of $f$ is killed by $N$.
We say that $f$ is an \textit{$N$-isomorphism} if it is an $N$-monomorphism and an $N$-epimorphism.
\end{definition}

\begin{proposition}\label{prop:adams}
Let $A$ be an abelian group equipped with endomorphisms $\psi^k$ for $k>0$ which commute with each other.
Suppose that there is a finite filtration
\[
	A=F^i \supset F^{i+1} \supset \cdots \supset F^j \supset F^{j+1}=0
\]
with $0\le i\le j$ consisting of subgroups preserved by $\psi^k$ such that $\psi^k$ acts on $F^p/F^{p+1}$ by the multiplication by $k^p$.
For $p\ge 0$, let $A^{(p)}$ be the subgroup of $A$ consisting of elements $x\in A$ such that $\psi^kx=k^px$ for all $k>0$.
Then:
\begin{enumerate}[label={\upshape(\roman*)}]
\item For $i\le p\le j$, $(\prod_{q=i}^{p-1}w_{p-q})A^{(p)}$ is in $F^p$ and the induced morphism
\[
	\prod_{q=i}^{p-1}w_{p-q}\colon A^{(p)}\to F^p/F^{p+1}
\]
is a $(\prod_{q=i}^jw_{|p-q|})$-isomorphism.
\item The canonical morphism
\[
	\bigoplus_{p=i}^j A^{(p)} \to A
\]
is a $(\prod_{i\le q,p\le j}w_{|p-q|})$-isomorphism.
\end{enumerate}
\end{proposition}
\begin{proof}
We follow \cite[2.8]{So85}.
Take integers $A_{pqk}$ ($p\ne q$) so that
\[
	w_{|p-q|}=\sum A_{pqk}(k^p-k^q).
\]
We fix $i\le p\le j$.
The morphism
\[
	\Phi_p:=\prod_{q=i}^{p-1}\Bigl(\sum_{k>1}A_{pqk}(\psi^k-k^q)\Bigr)
\]
sends $A$ to $F^p$, and on $A^{(p)}$ it is the multiplication by $\prod_{q=i}^{p-1}w_{p-q}$.
On the other hand, the morphism
\[
	\Psi_p:=\prod_{q=p+1}^j\Bigl(\sum_{k>1}A_{pqk}(\psi^k-k^q)\Bigr)
\]
sends $F^p$ to $A^{(p)}$, and modulo $F^{p+1}$ it is the multiplication by $\prod_{q=p+1}^jw_{q-p}$.
Moreover, this morphism kills $F^{p+1}$.
To summarize, $(\prod_{q=i}^{p-1}w_{p-q})A^{(p)}$ is in $F^p$ and the induced morphism
\[
	\prod_{q=i}^{p-1}w_{p-q}\colon A^{(p)} \to F^p/F^{p+1}
\]
is a $(\prod_{q=i}^jw_{|p-q|})$-isomorphism.

Next, we prove (ii).
Let us consider the commutative diagram
\[
\xymatrix{
	0 \ar[r] & \displaystyle\bigoplus_{p=i+1}^q A^{(p)}\cap F^{i+1} \ar[r] \ar[d] & \displaystyle\bigoplus_{p=i}^q A^{(p)} \ar[r] \ar[d] 
		& \displaystyle\Bigl(\bigoplus_{p=i+1}^q A^{(p)}/A^{(p)}\cap F^{i+1}\Bigr)\oplus A^{(i)} \ar[d] \ar[r] & 0 \\
	0 \ar[r] & F^{i+1} \ar[r] & A \ar[r] & A/F^{i+1} \ar[r] & 0
}
\]
with exact rows.
The left vertical arrow is a $(\prod_{i+1\le p, q\le j}w_{|p-q|})$-isomorphism by induction.
Since $\Phi_{i+1}$ sends $A$ to $F^{i+1}$ and it is the multiplication by $w_{p-i}$ on $A^{(p)}$, $\bigoplus_{p=i+1}^q A^{(p)}/A^{(p)}\cap F^{i+1}$ is killed by $\prod_{p=i+1}^jw_{p-i}$.
Combining it with (i), we see that the right vertical arrow is a $(\prod_{p=i}^jw_{p-i})^2$-isomorphism.
Consequently, the middle vertical arrow is a $(\prod_{i\le p,q\le j}w_{|p-q|})$-isomorphism.
\end{proof}

\subsection*{Example I: $\gamma$-filtration}

We refer to \cite{AT69} for the definition of (non-unital) special $\lambda$-rings, $\gamma$-filtrations and Adams operations.

\begin{lemma}
Let $I$ be a non-unital special $\lambda$-ring.
Assume that the $\gamma$-filtration on $I$ is finite.
Then the $\gamma$-filtration on $I$ together with the Adams operations satisfies the condition of Proposition \ref{prop:adams}.
\end{lemma}
\begin{proof}
This follows from [loc.\ cit., Proposition 5.3].
\end{proof}

\begin{definition}
Let $X$ be a scheme.
We define
\[
	\widetilde{K}(X) := \mbb{H}_{\mrm{Zar}}(X,B\GL^+)
\]
the global sections over $X$ of a Zariski-fibrant replacement of (a functorial model of) $B\GL^+$.
Let $Y$ be a closed subset of $X$ and $D$ a closed subscheme of $X$.
We define $\widetilde{K}^Y(X,D)$ to be the iterated homotopy fiber of the square
\[
\xymatrix{
	\widetilde{K}(X) \ar[r] \ar[d] & \widetilde{K}(D) \ar[d] \\
	\widetilde{K}(X\setminus Y) \ar[r] & \widetilde{K}(D\setminus Y).
}
\]
For a non-negative integer $n$, we write $\widetilde{K}_n^Y(X,D):=\pi_n\widetilde{K}^Y(X,D)$.
When $D=\varnothing$ (resp.\ $Y=X$), we write $\widetilde{K}^Y(X)$ (resp.\ $\widetilde{K}(X,D)$) for $\widetilde{K}^Y(X,D)$.
\end{definition}

\begin{lemma}\label{lem:lambda_relK}
Let $X$ be a scheme, $Y$ a closed subset of $X$ and $D$ a closed subscheme of $X$.
Then $\widetilde{K}^Y_n(X,D)$ is naturally a special $\lambda$-ring for each $n\ge 0$.
The first grading of the $\gamma$-filtration is
\[
	F^1_\gamma\widetilde{K}_n^Y(X,D)/F^2_\gamma\widetilde{K}_n^Y(X,D) \simeq
	\begin{cases}
	\mbb{H}^1_{\mrm{Zar},Y}((X,D),\mcal{O}^\times) & n=0 \\
	\mbb{H}^0_{\mrm{Zar},Y}((X,D),\mcal{O}^\times) & n=1 \\
	0 & n>1
	\end{cases}
\]
where $\mbb{H}^*_{\mrm{Zar},Y}((X,D),\mcal{O}^\times)$ denotes the homology of the iterated homotopy fiber of the square
\[
\xymatrix{
	\mbb{H}_{\mrm{Zar}}(X,\mcal{O}^\times) \ar[r] \ar[d] & \mbb{H}_{\mrm{Zar}}(D,\mcal{O}^\times) \ar[d] \\
	\mbb{H}_{\mrm{Zar}}(X\setminus Y,\mcal{O}^\times) \ar[r] & \mbb{H}_{\mrm{Zar}}(D\setminus Y,\mcal{O}^\times).
}
\]
\end{lemma}
\begin{proof}
Refer to \cite[Corollary 5.6]{Le97} for the fact $\widetilde{K}^Y_n(X,D)$ is a special $\lambda$-ring.
The first grading is calculated by the determinant $\det\colon B\GL^+\to B\mcal{O}^\times$.
Indeed, we have $\gamma^1+\gamma^2+\dotsb=0$ in $SK_*^Y(X,D)=\mbb{H}^{-*}_{\mrm{Zar},Y}((X,D),B\SL^+)$ as in \cite[p524]{So85}\footnotemark, and thus $F^2_\gamma SK_*^Y(X,D)=SK_*^Y(X,D)$.
\footnotetext{In \cite[Th\'eorem\`e 4]{So85}, it is asserted that $K_0(X)=H^0(X,\mbb{Z})\oplus \Pic(X)\oplus F^2_\gamma K_0(X)$, but it is not true.
The proof shows $\Gr^1_{\gamma}K_0(X)\simeq \Pic(X)$, but $\Pic(X)$ may not split.}
\end{proof}

\subsection*{Example II: coniveau filtration}

\begin{definition}\label{def:coniveau}
Let $X$ be a scheme and $Y$ a closed subset of $X$.
For each $p\ge 0$, we define
\[
	F^pK_0^Y(X) := \colim_Z\image(K_0^Z(X)\to K_0^Y(X))
\]
where $Z$ runs over all closed subset of $Y$ whose codimension in $X$ is greater or equal to $p$.
We call the filtration the \textit{coniveau filtration}.
We write $\Gr^pK_0^Y(X)$ for the $p$-th grading of the coniveau filtration.
\end{definition}

\begin{lemma}\label{lem:coniveau}
Let $X$ be a regular scheme of dimension $d$ and $Y$ a closed subset of $X$ of codimension $p$.
Then the coniveau filtration
\[
	K_0^Y(X)=F^p \supset F^{p+1} \supset \dotsb \supset F^d \supset F^{d+1}=0
\]
together with the Adams operations satisfies the condition of Proposition \ref{prop:adams}.
\end{lemma}
\begin{proof}
First, we prove the case $Y$ is a closed point of $X$.
Let $j$ be the inclusion $Y\hookrightarrow X$ and $j_*$ the pushfoward $K_0(Y)\to K^Y_0(X)$.
Then, by \cite[Th\'eor\`eme 3]{So85}, we have $\psi^k\circ j_*=k^d(j_*\circ\psi^k)$.
Since $\psi^k$ is the identity on $K_0(Y)\simeq\mbb{Z}$ and $j_*$ is an isomorphism, we conclude that $\psi^k$ acts by the multiplication by $k^d$ on $K^Y(X)$.
This proves the case $Y$ is a closed point.

We prove the remaining case by descending induction on $p$.
We have seen the case $p=d$.
Let $p<d$.
If $q>p$, then the canonical morphism
\[
	\colim_Z \Gr^qK_0^Z(X) \to \Gr^qK_0^Y(X),
\]
where $Z$ runs over all closed subsets of $Y$ whose codimension in $X$ is greater or equal to $p+1$, is surjective.
By induction, the Adams operation $\psi^k$ acts by the multiplication by $k^q$ on the left term, and so on the right.
It remains to show that the Adams operation $\psi^k$ acts by the multiplication by $k^p$ on $\Gr^pK_0^Y(X)=K_0^Y(X)/F^{p+1}K_0^Y(X)$.
This follows from the exact sequence \cite[Lemma 5.2]{GS87}
\[
\xymatrix@1{
	0 \ar[r] & F^{p+1}K_0^Y(X) \ar[r] & K_0^Y(X) \ar[r] & \displaystyle\bigoplus_{y\in Y\cap X^{(p)}}K^y_0(X_y) \ar[r] & 0.
}
\]
Note that the Adams operations act on the sequence and we have seen that $\psi^k$ acts on the right term by the multiplication by $k^p$.
This completes the proof.
\end{proof}

\subsection*{Example III: relative coniveau filtration}

\begin{definition}\label{def:rel_coniveau}
Let $X$ be a scheme and $D$ a closed subscheme of $X$.
For each $p\ge 0$, we define
\[
	F^pK_0(X,D) := \colim_Z\image(K_0^Z(X)\to K_0(X,D))
\]
where $Z$ runs over all closed subset in $X$ not meeting $D$ of codimension greater or equal to $p$.
We call the filtration the \textit{relative coniveau filtration}.
We write $\Gr^pK_0(X,D)$ for the $p$-th grading of the relative coniveau filtration.
\end{definition}

\begin{lemma}\label{lem:rel_coniveau}
Let $X$ be a scheme of dimension $d$ with an ample family of line bundles and $D$ a closed subschme of $X$.
Assume that $X\setminus D$ is regular and that $D$ has an affine open neighborhood in $X$.
Then the relative coniveau filtration
\[
	K_0(X,D)=F^0 \supset F^1 \supset \dotsb \supset F^d \supset F^{d+1}=0
\]
together with the Adams operations satisfies the condition of Proposition \ref{prop:adams}.
\end{lemma}
\begin{proof}
By Lemma \ref{lem:surjectivity}, $F^0K_0(X,D)=K_0(X,D)$.
By definition, the canonical morphism
\[
	\colim_{Y\in\mcal{S}(X|D)}\Gr^pK_0^Y(X) \to \Gr^pK_0(X,D)
\]
is surjective.
The morphism is compatible with the Adams operations, and thus the result follows from Lemma \ref{lem:coniveau}.
\end{proof}

\section{Proof of Theorem \ref{mainthm}}\label{S:proof}

\subsection*{Cycle class morphisms with supports}

\begin{definition}\label{def:cyc}
Let $X$ be a regular scheme, $Y$ a closed subset of $X$ and $p$ an integer.
The \textit{cycle class morphism} is a group morphism
\[
	\msf{cyc}\colon \mcal{Z}^p_Y(X) \to K_0^Y(X)
\]
defined by sending the closure $V$ of a point $v\in X^{(p)}\cap Y$ to $\mcal{O}_{V}$.
\end{definition}

\begin{theorem}[Gillet-Soul\'e]\label{thm:GS}
Let $X$ be a regular scheme, $Y$ a closed subset of $X$ and $p$ an integer.
Then the cycle class morphism induces a surjective group morphism
\[
	\msf{cyc}\colon \CH^p_Y(X) \twoheadrightarrow F^pK_0^Y(X)/F^{p+1}K_0^Y(X)
\]
and its kernel is $(\prod_{i=1}^{p-2}\omega_i)$-torsion.
\end{theorem}
\begin{proof}
This is essentially a consequence of results in \cite{So85} as observed in \cite[Theorem 8.2]{GS87}.
Since our formulation claims a little bit stronger than the original one, we give a sketch of the proof here.

Consider the Gersten-Quillen spectral sequence
\[
	E^{p,q}_1 = \bigoplus_{x\in X^{(p)}\cap Y}K_{-p-q}(k(x)) \Rightarrow K^Y_{-p-q}(X).
\]
The Adams operations act on this spectral sequence, and by the Riemann-Roch type formula \cite[Th\'eor\`eme 3]{So85} $\psi^k$ acts by the multiplication by $k^{p+i}$ on $E^{p,-p-i}_r$ for $i=0,1$ and $r\ge 1$.
It follows that the differential $d_r^p \colon E^{p-r,-p+r-1}_r \to E^{p,-p}_r$ is killed by $\omega_{r-1}$ for $p\ge r>1$.
In fact, $d_r^r=0$ for $r>1$, because
\[
	E^{0,-1}_2 \simeq \bigoplus_{x\in X^{(0)}\cap Y}\mcal{O}(\overline{\{x\}})^\times \simeq E^{0,-1}_\infty.
\]
Consequently, the kernel of the canonical surjection $E^{p,-p}_2 \twoheadrightarrow E^{p,-p}_\infty$ is killed by $\prod_{i=1}^{p-2}\omega_i$.
On the other hand, we have $E^{p,-p}_2 \simeq \CH^p_Y(X)$ and $E^{p,-p}_\infty \simeq \Gr^pK^Y_0(X)$.
Hence, we get the result.
\end{proof}

\begin{lemma}\label{lem:commutativity}
Let $X$ be a regular scheme, $Y$ a closed subset of $X$, $D$ a principal effective Cartier divisor on $X$ and $p$ an integer.
We denote the inclusion $D\hookrightarrow X$ by $\iota$.
Then the diagram
\[
\xymatrix{
	K_0^Y(X) \ar[r]^-{\iota^*} & K_0^{Y\cap D}(D)/F^{p+1}K_0^{Y\cap D}(D) \\
	\mcal{Z}^p_{Y}(X) \ar[r]^-{\iota^*} \ar[u]^{\msf{cyc}} & \CH^p_{Y\cap D}(D) \ar[u]^{\msf{cyc}}
}
\]
commutes.
Here, the bottom horizontal map $\iota^*$ is the one defined in Construction \ref{constr:pullback}.
\end{lemma}
\begin{proof}
Suppose that we are given $v\in X^{(p)}\cap Y$ and denote its closure in $X$ by $V$.
If $V\nsubseteq D$, then the commutativity is clear, i.e., $(\iota^*\circ\msf{cyc})[V]=(\msf{cyc}\circ\iota^*)[V]$.
Suppose that $V\subseteq D$.
Then $\iota^*[V]=0$ in $\CH^p_{Y\cap D}(D)$.
Also, $\iota^*(\mcal{O}_V) = [\mcal{O}_V \xrightarrow{f} \mcal{O}_V] = 0$ in $K_0^{Y\cap D}(D)$, where $f$ denotes the defining equation of $D$.
This proves the lemma.
\end{proof}

%\noindent
%\textcolor{magenta}{\textbf{Comment.}
%Can we remove the assumption that $D$ is principal?}

\begin{corollary}\label{cor:commutativity}
Under the situation in Lemma \ref{lem:commutativity}, the restriction morphism $K_0^Y(X) \to K_0^{Y\cap D}(D)$ preserves the coniveau filtration.
\end{corollary}

\subsection*{On codimension one}

\begin{lemma}\label{lem:cod1}
Let $X$ be a regular scheme and $Y$ a closed subset of $X$ not containing any irreducible components of $X$.
Then there are natural isomorphisms
\begin{gather*}
	K_0^Y(X) \simeq \CH^1_Y(X) \oplus F^2K_0^Y(X), \\
	\CH^1_Y(X) \simeq H^1_{\mrm{Zar},Y}(X,\mcal{O}^\times)
	\quad \text{and} \quad
	F^2K_0^Y(X) \simeq F^2_{\gamma}K_0^Y(X).
\end{gather*}
\end{lemma}
\begin{proof}
Since $K^Y_0(X)=F^1K_0^Y(X)$, we get an isomorphism
\[
	\msf{cyc}\colon \CH^1_Y(X) \xrightarrow{\sim} K_0^Y(X)/F^2K_0^Y(X)
\]
by Theorem \ref{thm:GS}.
Since $\CH^1_Y(X)$ is the free abelian group generated by the irreducible components of $Y$ of codimension one in $X$, the above isomorphism factors through $K_0^Y(X)$.
This proves the first isomorphism.
The second isomorphism follows from the quasi-isomorphism $\mcal{Z}^1(X,\bullet)\simeq \mcal{O}^\times[1]$ (\cite[Theorem 6.1]{Bl86}) and Lemma \ref{lem:lambda_relK}.
The last isomorphism follows from the first two isomorphisms.
\end{proof}

\begin{corollary}\label{cor:cod1}
Let $X$ be a regular scheme and $D$ an effective Cartier divisor on $X$.
Assume that $D$ has an affine open neighborhood in $X$.
Then there are exact sequences
\begin{gather*}
\xymatrix@1{
	\displaystyle\colim_{Y\in\mcal{S}(X|D,1)}F^2K_0^Y(X\times\Box^1) \ar[r]^-{\iota_0^*-\iota_1^*} & \displaystyle\colim_{Y\in\mcal{S}(X|D)}F^2K_0^Y(X) \ar[r]^-\epsilon & F^2K_0(X,D) \ar[r] & 0
} \\
\xymatrix@1{
	\displaystyle\colim_{Y\in\mcal{S}(X|D,1)}\Gr^1K_0^Y(X\times\Box^1) \ar[r]^-{\iota_0^*-\iota_1^*} & \displaystyle\colim_{Y\in\mcal{S}(X|D)}\Gr^1K_0^Y(X) \ar[r]^-\epsilon & \Gr^1K_0(X,D) \ar[r] & 0
}
\end{gather*}
and the same for the $\gamma$-filtration.
\end{corollary}
\begin{proof}
The morphisms are well-defined by Corollary \ref{cor:commutativity}.
We may assume that $X$ is connected.
If $D$ is empty, then the result is clear. %homotopy invariance.
If $D$ is not empty, then it follows from Theorem \ref{thm:K-theory} and Lemma \ref{lem:cod1}.
\end{proof}

\begin{theorem}\label{thm:cod1}
Let $X$ be a regular scheme and $D$ an effective Cartier divisor on $X$.
Assume that $D$ has an affine open neighborhood in $X$.
Then there are natural isomorphisms
\[
	\CH^1(X|D) \simeq \Gr^1K_0(X,D) \simeq \Gr^1_{\gamma}K_0(X,D) \simeq \Pic(X,D).
\]
\end{theorem}
\begin{proof}
The first isomorphism follows from the commutative diagram
\[
\xymatrix{
	\displaystyle\colim_{Y\in\mcal{S}(X|D,1)}\Gr^1K_0^Y(X\times\Box^1) \ar[r]^-{\iota_0^*-\iota_1^*} 
		& \displaystyle\colim_{Y\in\mcal{S}(X|D)}\Gr^1K_0^Y(X) \ar[r]^-\epsilon & \Gr^1K_0(X,D) \ar[r] & 0 \\
	\displaystyle\colim_{Y\in\mcal{S}(X|D,1)}\mcal{Z}^1_Y(X\times\Box^1) \ar[r]^-{\iota_0^*-\iota_1^*} \ar@{->>}[u]^{\msf{cyc}}
		& \displaystyle\colim_{Y\in\mcal{S}(X|D)}\CH^1_Y(X) \ar[r]^-\epsilon \ar[u]^{\msf{cyc}}_{\simeq} & \CH^1(X|D) \ar[r] \ar@{.>}[u] & 0.
}
\]
This is indeed commutative by Lemma \ref{lem:commutativity}, and the rows are exact by Proposition \ref{prop:cycle} and Corollary \ref{cor:cod1}.
The middle vertical arrow is an isomorphism by Theorem \ref{thm:GS}.

The second isomorphism follows from Lemma \ref{lem:cod1} and Corollary \ref{cor:cod1}.
The last isomorphism has been observed in Lemma \ref{lem:lambda_relK}.
\end{proof}

\begin{example}\label{counterexample}
Let $k$ be a field, $X=\mbb{P}^1_k\times_k\mbb{P}^1$ and $D=\mbb{P}^1_k$ which we regard as a Cartier divisor on $X$ by the diagonal embedding.
Then $\CH^1(X|D)=0$, but $\Gr^1_{\gamma}K_0(X,D)=\Pic(X,D)=\mbb{Z}$.
\end{example}

\subsection*{On higher codimension}

\begin{lemma}\label{lem:refine}
Let $X$ be a regular scheme, $D$ an effective Cartier divisor on $X$ and $p$ an integer.
Assume that $D$ has an affine open neighborhood in $X$.
Suppose we are given
\[
	\alpha \in \ker\Bigl(\colim_{Y\in\mcal{S}(X|D)}\Gr^pK_0^Y(X) \xrightarrow{\epsilon} \Gr^pK_0(X,D)\Bigr).
\]
Then there exists $\beta\in\Gr^pK_0^W(X\times\Box^1)$ for some $W\in\mcal{S}(X|D,1)$ such that
\[
	\iota_0^*\beta-\iota_1^*\beta = \Bigl(\prod_{i=2}^{p-1}w_{p-i}\Bigr)\Bigl(\prod_{2\le i,j\le p}w_{|i-j|}\Bigr)^2\alpha.
\]
\end{lemma}
\begin{proof}
We may assume that $D$ is non-empty.
The case $p\le 1$ is true by Theorem \ref{thm:K-theory} and Lemma \ref{cor:cod1}.
Let $p>1$.
We consider the diagram
\[
\xymatrix{
	\displaystyle\colim_{Y\in\mcal{S}(X|D,1)}\bigoplus_{i=2}^p(F^2K_0^Y(X\times\Box^1)/F^{p+1})^{(i)} \ar[d] \ar[r] 
		& \displaystyle\colim_{Y\in\mcal{S}(X|D)}\bigoplus_{i=2}^p(F^2K_0^Y(X)/F^{p+1})^{(i)} \ar[d] \\
	\displaystyle\colim_{Y\in\mcal{S}(X|D,1)}F^2K_0^Y(X\times\Box^1)/F^{p+1} \ar[r]
		& \displaystyle\colim_{Y\in\mcal{S}(X|D)}F^2K_0^Y(X)/F^{p+1} \\
	\displaystyle\colim_{Y\in\mcal{S}(X|D,1)}\Gr^pK_0^Y(X\times\Box^1) \ar[r] \ar@{^{(}->}[u]
		& \displaystyle\colim_{Y\in\mcal{S}(X|D)}\Gr^pK_0^Y(X) \ar@{^{(}->}[u]
}
\]
Suppose given $\alpha$ as in the statement.
According to Corollary \ref{cor:cod1}, there exists $\tilde{\beta}\in F^2K_0^W(X\times\Box^1)/F^{p+1}$ for some $W\in\mcal{S}(X|D,1)$ such that
\[
	\iota_0^*\beta-\iota_1^*\beta = \alpha  \quad \text{in} \quad \colim_{Y\in\mcal{S}(X|D)}F^2K_0^Y(X)/F^{p+1}.
\]

By Lemma \ref{prop:adams}, $(\prod_{2\le i,j\le p}w_{|i-j|})\tilde{\beta}$ lifts to 
\[
	\beta^{(2)}+\beta^{(3)}+\dotsb+\beta^{(p)} \in \bigoplus_{i=2}^p(F^2K_0^W(X\times\Box^1)/F^{p+1})^{(i)}.
\]
By Lemma \ref{prop:adams} again, we see that
\[
	\Bigl(\prod_{2\le i,j\le p}w_{|i-j|}\Bigr)(\iota_0^*\beta^{(p)}-\iota_1^*\beta^{(p)}) = \Bigl(\prod_{2\le i,j\le p}w_{|i-j|}\Bigr)^2\alpha.
\]
Since $(\prod_{i=2}^{p-1}w_{p-i})\beta^{(p)}$ is in $\Gr^pK_0^W(X\times\Box^1)$, we are done.
\end{proof}

\begin{proof}[Proof of Theorem \ref{mainthm}]
Let $X$ be a regular scheme, $D$ an effective Cartier divisor on $X$ and $p$ an integer.
Let us consider the commutative diagram
\[
\xymatrix{
	\displaystyle\colim_{Y\in\mcal{S}(X|D,1)}\Gr^pK_0^Y(X\times\Box^1) \ar[r]
		& \displaystyle\colim_{Y\in\mcal{S}(X|D)}\Gr^pK_0^Y(X) \ar[r] & \Gr^pK_0(X,D) \ar[r] & 0 \\
	\displaystyle\colim_{Y\in\mcal{S}(X|D,1)}\CH^p_Y(X\times\Box^1) \ar[r] \ar@{->>}[u]^{\msf{cyc}_1}
		& \displaystyle\colim_{Y\in\mcal{S}(X|D)}\CH^p_Y(X) \ar[r] \ar@{->>}[u]^{\msf{cyc}_0} & \CH^p(X|D) \ar@{.>>}[u]^{\msf{cyc}_{\mrm{rel}}} \ar[r] & 0.
}
\]
Since the bottom row is exact (Proposition \ref{prop:cycle}) and the composite of the first two morphisms in the upper row is zero (Theorem \ref{thm:K-theory}), a morphism $\msf{cyc}_{\mrm{rel}}$ is induced and it is surjective.

Suppose we are given $\alpha\in\CH^p(X|D)$ such that $\msf{cyc}_{\mrm{rel}}(\alpha)=0$.
By Lemma \ref{lem:refine} and by a simple diagram chase, there exists $\beta\in\ker(\msf{cyc}_0)$ which lifts
\[
	\Bigl(\prod_{i=2}^{p-1}w_{p-i}\Bigr)\Bigl(\prod_{2\le i,j\le p}w_{|i-j|}\Bigr)^2\alpha.
\]
By Theorem \ref{thm:GS}, $(\prod_{i=1}^{p-2}w_i)\beta=0$.

The last statement (comparison between $F^*$ and the gamma filtration) follows from Proposition \ref{prop:adams} and Theorem \ref{thm:cod1}.
\end{proof}

\section{Applications}

\subsection*{Multiplicative structure on Chow groups with modulus}

As an application of Theorem \ref{mainthm}, we prove that there is a natural multiplicative structure on $\CH^*(X|D)$ up to torsion.
We formulate it keeping track of the changes of $D$.
Note that the Chow groups with modulus yield a contravariant functor
\[
	\CH^*(X|-)\colon \msf{Div}^+(X)^{\mrm{op}} \to \msf{GrAb}
\] 
from the category of effective Cartier divisors on $X$ to the category of graded abelian groups.

\begin{notation}
Let $\mcal{A}$ be an additive category and $l$ a positive integer.
We define a category $\mcal{A}_{\mbb{Z}[1/l]}$ having the same objects as $\mcal{A}$ and $\Hom_{\mcal{A}[1/l]}(-,-)=\Hom_{\mcal{A}}(-,-)\otimes_{\mbb{Z}}\mbb{Z}[1/l]$.
For an object $M$ in $\mcal{A}$, we denote its image in $\mcal{A}_{\mbb{Z}[1/l]}$ by $M_{\mbb{Z}[1/l]}$.
\end{notation}

\begin{theorem}\label{thm:mult}
Let $X$ be a regular scheme of dimension $d$.
Let $\msf{Div}^+_{\mrm{aff}}(X)$ be the full subcategory of $\msf{Div}^+(X)$ consisting of divisors admitting affine open neighborhoods in $X$. 
Then 
\[
	\CH^*(X|-)_{\mbb{Z}[1/(d-1)!]} \in \msf{Fun}(\msf{Div}^+_{\mrm{aff}}(X)^{\mrm{op}},\msf{GrAb})_{\mbb{Z}[1/(d-1)!]}
\]
has a natural commutative monoid structure.
\end{theorem}
\begin{proof}
By Theorem \ref{mainthm}, it suffices to show that $\Gr^*_{\gamma}K_0(X,-)$ has a commutative monoid structure, but it is obvious from its definition.
\end{proof}

\subsection*{Chow groups with topological modulus}

Here, we show that if $D$ is $K_1$-regular then the Chow group with modulus $\CH^*(X|D)$ becomes much simpler (at least up to torsion).
Compare the following definition with Notation \ref{not:modulus} and Definition \ref{def:Chow_modulus}.

\begin{definition}
Let $X$ be a unicodimensional catenary scheme, $D$ an effective Cartier divisor and $p$ an integer.
\begin{enumerate}[(1)]
\item $\mcal{S}(X|D_{\mrm{top}}):=\mcal{S}(X|D)$ and $\mcal{Z}^p(X|D_{\mrm{top}}):=\mcal{Z}^p(X|D)$.
\item $\mcal{S}(X|D_{\mrm{top}},1)$ is the set of all closed subsets of $X\times\Box^1$ not meeting $D\times\Box^1$.
\item $\mcal{Z}^p(X|D_{\mrm{top}},1)$ is the free abelian group with generators $[W]$, one for each $w\in (X\times\Box^1)^{(p)}$ whose closure $W$ is dominant over $\Box^1$ and not meeting $D\times\Box^1$.
\item We define 
\[
	\CH^p(X|D_{\mrm{top}}) := \coker\bigl(\mcal{Z}^p(X|D_{\mrm{top}},1) \xrightarrow{\iota_0^*-\iota_1^*} \mcal{Z}^p(X|D_{\mrm{top}})\bigr).
\]
\end{enumerate}
\end{definition}

\begin{remark}
The groups $\CH^p(X|D_{\mrm{top}})$ and its higher variant have been studied in \cite{Mi17,IK} by the name of \textit{na\"ive Chow groups with modulus} and \textit{Chow groups with topological modulus} respectively.
\end{remark}

\begin{lemma}\label{lem:cycle_top}
Let $X$ be a unicodimensional catenary scheme, $D$ an effective Cartier divisor on $X$ and $p$ an integer.
Then the sequence 
\[
\xymatrix@1{
	\displaystyle\colim_{Y\in\mcal{S}(X|D_{\mrm{top}},1)}\mcal{Z}^p_Y(X\times\Box^1) \ar[r]^-{\iota_0^*-\iota_1^*} 
	& \displaystyle\colim_{Y\in\mcal{S}(X|D_{\mrm{top}})}\CH^p_Y(X) \ar[r]^-{\epsilon} & \CH^p(X|D_{\mrm{top}}) \ar[r] & 0
}
\]
is exact.
\end{lemma}
\begin{proof}
Same as Proposition \ref{prop:cycle}.
\end{proof}

The following is a variant of Theorem \ref{thm:K-theory}, which is a special case of \cite[Lemma xx]{IK}.

\begin{lemma}\label{lem:K-theory_top}
Let $X$ be a scheme and $D$ an effective Cartier divisor on $X$ admitting an affine open neighborhood in $X$.
Assume that $X$ is $K_0$-regular and that $D$ is $K_1$-regular.
Then the sequence
\[
\xymatrix@1{
	\displaystyle\colim_{Y\in\mcal{S}(X|D_{\mrm{top}},1)}K_0^Y(X\times\Box^1) \ar[r]^-{\iota_0^*-\iota_1^*} 
	& \displaystyle\colim_{Y\in\mcal{S}(X|D_{\mrm{top}})}K_0^Y(X) \ar[r]^-{\epsilon} & K_0(X,D) \ar[r] & 0
}
\]
is exact.
\end{lemma}
\begin{proof}
By the assumption, $K_0(X,D)\simeq K_0(X\times\Box^1,D\times\Box^1)$, from which it follows that the composite $\epsilon\circ(\iota_0^*-\iota_1^*)$ is zero.
The surjectivity of $\epsilon$ follows from Lemma \ref{lem:surjectivity}, and the exactness at the middle term follows from Lemma \ref{lem:K_1}.
\end{proof}

\begin{proposition}\label{prop:top}
Let $X$ be a regular scheme and $D$ an effective Cartier divisor on $X$.
Assume that $D$ is $K_1$-regular and admits an affine open neighborhood in $X$.
Then, for each integer $p$, there exists a surjective group morphism
\[
	\CH^p(X|D_{\mrm{top}}) \twoheadrightarrow F^pK_0(X,D)/F^{p+1}K_0(X,D)
\]
such that its kernel is $(p-1)!^N$-torsion for some positive integer $N$ depending only on $p$.
\end{proposition}
\begin{proof}
This follows from Lemma \ref{lem:cycle_top} and Lemma \ref{lem:K-theory_top} as in \S\ref{S:proof}.
\end{proof}

\begin{theorem}
Let $X$ be a regular scheme and $D$ an effective Cartier divisor on $X$.
Assume that $D$ is $K_1$-regular and admits an affine open neighborhood in $X$.
Then, for each integer $p$, the canonical morphism
\[
	\CH^p(X|D) \to \CH^p(X|D_{\mrm{top}})
\]
is a $(p-1)!^N$-isomorphism for some positive integer $N$ depending only in $p$.
\end{theorem}
\begin{proof}
This follows from Theorem \ref{mainthm} and Proposition \ref{prop:top}.
\end{proof}

\end{document}